\documentclass{amsart}
\usepackage{latexsym,amsmath,amssymb}
\newtheorem{thm}{Theorem}[section]
\newtheorem{cor}[thm]{Corollary}

\newtheorem{lem}[thm]{Lemma}
\newtheorem{prop}[thm]{Proposition}
\theoremstyle{definition}\newtheorem{defn}[thm]{Definition}
\theoremstyle{remark}

\numberwithin{equation}{section}

\begin{document}

\title[]
{On the Deddens algebras of a class of bounded operators
}

\author{\sc\bf  Z. Huang, Y. Estaremi and S. Shimi}
\address{ \sc Z. Huang}
\email{jameszhuang923@gmail.com}
\address{Huxley Building Department of Mathematics, South Kensington Campus,Imperial College London, London, UK}
\address{\sc Y. Estaremi}
\email{y.estaremi@gu.ac.ir}
\address{Department of Mathematics, Faculty of Sciences, Golestan University, Gorgan, Iran.}
\address{\sc S. Shimi}
\email{S.shimi@gnu.ac.ir}
\address{Department of Mathematics, Gu university, Iran.}

\thanks{}

\thanks{}



\date{}

\dedicatory{}

\commby{}

\begin{abstract}
In this paper, we investigate the relation between the Deddens and spectral radius algebras of two bounded linear operators, noting a similarity between them. Additionally, we characterize the Deddens and spectral radius algebras related to rank one operators, operators that are similar to rank one operators, operators that are majorized by rank one operators, and quasi-isometry operators. Furthermore, we apply these results to the class of weighted conditional type operators on the Hilbert space $L^2(\mu)$.

\end{abstract}

\maketitle

\section{ \sc\bf Introduction and Preliminaries}
Let $\mathcal{H}$ be a complex Hilbert spaces, $\mathcal{B}(\mathcal{H})$ be the Banach algebra of all bounded linear operators on $\mathcal{H}$, where $I=I_{\mathcal{H}}$ is the identity operator on $\mathcal{H}$. If $T\in \mathcal{B}(\mathcal{H})$, then $T^*$ is the adjoint of $T$.

Let $\mathcal{C}$ be a class of operators on the Hilbert space $\mathcal{H}$, and let $T\in \mathcal{B}(\mathcal{H})$. We say that $T$ is similar to an element of $\mathcal{C}$ if there exists $C\in \mathcal{C}\cap \mathcal{B}(\mathcal{H})$ and an invertible operator $A\in \mathcal{B}(\mathcal{H})$ such that $AT=CA$. In this case, we say that $A$ is a similarity between $T$ and $C$, or $T$ is similar to $C$ by $A$. Since $A$ is invertible, we have $T=A^{-1}CA$ and $C=ATA^{-1}$.

Recall that a bounded linear operator $T$ is called quasi-normal if $T$ commutes with $T^*T$, i.e., $TT^*T=T^*TT$.

Let $A$ and $T$ be operators in $\mathcal{B}(\mathcal{H})$, where $A$ is a non-zero positive operator. The operator $T$ is called an $A$-isometry if $T^*AT=A$
. It is easy to see that if $T$ is an $A$-isometry, then $T^n$ is also an $A$-isometry for every $n\in \mathbb{N}$. In order to \cite{cs}, for $n\in \mathbb{N}$, we say that $T$ is an $n$-quasi-isometry if $T$ is a $T^{*^n}T^n$-isometry. Hence, $T$ is an $n$-quasi-isometry if and only if $T$ is an isometry on $\mathcal{R}(T^n)$. Moreover, $T$ is called a quasi-isometry if it is a 1-quasi-isometry.

If $A$ is an invertible operator in $\mathcal{B}(\mathcal{H})$, then the collection

$$\{T\in \mathcal{B}(\mathcal{H}): \sup_{n\in \mathbb{N}}\|A^nTA^{-n}\|<\infty\}$$
is called the Deddens algebra of $A$ and denoted by $\mathcal{D}_A$. It is easy to see that
 $S\in \mathcal{D}_T$ if and only if there exists $M>0$ such that
$$\|T^nSx\|\leq M\|T^nx\|, \ \ \forall n\in \mathbb{N}, \ \ x\in \mathcal{H}.$$

Let $T\in \mathcal{B}(\mathcal{H})$ and $r(T)$ be the spectral radius of $T$. For $m\geq1$, we define
\begin{equation}\label{e1}
R_m(T)=R_m:=\left(\sum^{\infty}_{n=0}d^{2n}_mT^{\ast ^n}T^n\right)^{\frac{1}{2}},
\end{equation}
where $d_m=\frac{1}{1/m+r(T)}$. Since $d_m\uparrow 1/r(T)$, the sum in \ref{e1} is norm convergent and the operators $R_m$ are well-defined, positive, and invertible. The spectral radius algebra $\mathcal{B}_T$ of $T$ consists of all operators $S\in \mathcal{B}(\mathcal{H})$ such that
$$\sup_{m\in \mathbb{N}}\|R_mSR^{-1}_m\|<\infty,$$
or equivalently, $S\in \mathcal{B}_T$ if and only if there exists $M>0$ such that
$$\sum^{\infty}_{n=0}d^{2n}_m\|T^nSx\|\leq M\sum^{\infty}_{n=0}d^{2n}_m\|T^nx\|, \forall m\in \mathbb{N}, \ \ \forall x\in \mathcal{H}.$$

The set $\mathcal{B}_T$ is an algebra and it contains all operators that commute with $T$ ($\{T\}'$). By the above definitions, for each $T\in \mathcal{B}(\mathcal{H})$, we have
$$\{T\}'\subseteq \mathcal{D}_T\subseteq \mathcal{B}_T.$$

The spectral radius and Deddens algebras help us to find invariant and hyperinvariant subspaces for a bounded linear operator. Many mathematicians have investigated the problem of finding invariant subspaces for special classes of bounded linear operators by studying the invariant subspaces of spectral radius and Deddens algebras. The latest results on Deddens algebras can be found in \cite{jm}. For more information one can see \cite{blpw,ej1,ja,ml,lp,pe,si}. In this paper, we are concerned with the Deddens and spectral radius algebras of some classes of bounded linear operators on Hilbert spaces. In Section 2, we first find the relation between Deddens and spectral radius algebras of two bounded linear operators that have a similarity between them. We also investigate the Deddens and spectral radius algebras of rank one operators and operators that are majorized by a rank one operator. In the sequel, we obtain the Deddens and spectral radius algebras of quasi-isometry operators. In Section 3, we apply the results of Section 2 to the class of WCT operators on the Hilbert space $L^2(\mu)$.

\section{ \sc\bf Deddens and spectral radius algebras of similar operators}
In this section we first investigate the relation between the Deddens and spectral radius algebras of two bounded linear operators that are similar through an invertible operator.

\begin{prop}\label{p0}
Let $T, A, C\in\mathcal{B}(\mathcal{H})$ and $A$ be invertible such that $T$ is similar to $C$ by $A$. Then $T^n$ is similar to $C^n$ by $A$, for every $n\in \mathbb{N}$ and $A\mathcal{D}_T=\mathcal{D}_CA$ (or equivalently $A\mathcal{D}_TA^{-1}=\mathcal{D}_C$, $\mathcal{D}_T=A^{-1}\mathcal{D}_CA$). Also, $A\mathcal{B}_T=\mathcal{B}_{C}A$ (or equivalently $A\mathcal{B}_TA^{-1}=\mathcal{B}_C$, $\mathcal{B}_T=A^{-1}\mathcal{B}_CA$).
\end{prop}
\begin{proof}
Let $S\in \mathcal{D}_T$, then there exists $M>0$ such that
$$ \|T^nSx\|\leq M\|T^nx\|, \ \ \forall n\in \mathbb{N}, \ \ x\in \mathcal{H}.$$
So we have
\begin{align*}
\|C^nASA^{-1}x\|&=\|AT^nSA^{-1}x\|\\
&\leq \|A\|\|T^nSA^{-1}x\|\\
&M\|A\|\|T^nA^{-1}x\|\\
&M\|A\|\|A^{-1}C^nx\|\\
&M\|A\|\|A^{-1}\|\|C^nx\|,
\end{align*}
for all $n\in \mathbb{N}, \ \ x\in \mathcal{H}$.
Thus we have $ASA^{-1}\in \mathcal{D}_C$ and consequently $A\mathcal{D}_TA^{-1}\subseteq\mathcal{D}_C$. Similarly we get the converse i.e., $\mathcal{D}_C\subseteq A\mathcal{D}_TA^{-1}$ and so $A\mathcal{D}_TA^{-1}=\mathcal{D}_C$.\\
Let $S\in \mathcal{B}_T$, then
there exists $M>0$ such that
$$\sum^{\infty}_{n=0}d^{2n}_m\|T^nSx\|\leq M\sum^{\infty}_{n=0}d^{2n}_m\|T^nx\|, \forall m\in \mathbb{N}, \ \ \forall x\in \mathcal{H}.$$
Hence we have
\begin{align*}
\sum^{\infty}_{n=0}d^{2n}_m\|C^nASA^{-1}x\|&=\sum^{\infty}_{n=0}d^{2n}_m\|AT^nSA^{-1}x\|\\
&\leq\|A\|\sum^{\infty}_{n=0}d^{2n}_m\|T^nSA^{-1}x\|\\
&\leq M\|A\|\sum^{\infty}_{n=0}d^{2n}_m\|T^nA^{-1}x\|\\
&= M\|A\|\sum^{\infty}_{n=0}d^{2n}_m\|A^{-1}C^nx\|\\
&\leq M\|A\|\|A^{-1}\|\sum^{\infty}_{n=0}d^{2n}_m\|C^nx\|,
\end{align*}
for all $m\in \mathbb{N}$ and for all $x\in \mathcal{H}$. This means that $ASA^{-1}\in \mathcal{B}_C$ and so $A\mathcal{B}_TA^{-1}\subseteq\mathcal{B}_C$. Similarly one can prove the converse. Therefore $A\mathcal{B}_TA^{-1}=\mathcal{B}_C$.

\end{proof}
Let $X, Y, Z$ be Banach spaces and $\mathcal{B}(X,Y)$ be the Banach space of all bounded linear operators from $X$ into $Y$. Also, $\mathcal{R}(T)$, $\mathcal{N}(T)$ are the range and the kernel of $T$, respectively. If $T\in \mathcal{B}(X,Y)$ and $S\in \mathcal{B}(X,Z)$, then we say that $T$ majorizes $S$ if there exists $M>0$ such that
$$\|Sx\|\leq M\|Tx\|, \ \ \ \ \ \text{for all} \ x\in X.$$
The following characterization are known in the case of Hilbert spaces.
\begin{thm}\cite{do}\label{t0} For  $T, S\in \mathcal{B}(\mathcal{H})$, the following conditions are equivalent:\\
(1) $\mathcal{R}(S)\subseteq \mathcal{R}(T)$;\\
(2) $T^*$ majorizes $S^*$;\\
(3) $S=TU$ for some $U\in \mathcal{B}(\mathcal{H})$.
\end{thm}

For $x, y \in \mathcal{H}$, we have $x\otimes y \in \mathcal{B}(\mathcal{H})$, and $\|x\otimes y\| = \|x\|\|y\|$. Here, $(x\otimes y)h = \langle h,y\rangle_{\mathcal{H}} x$ for every $h\in \mathcal{H}$. It is known that all rank one operators are of the form $x\otimes y$, and as a result, they generate finite rank operators on $\mathcal{H}$.

In the following theorem, we determine the elements of the dense algebras of rank one operators on Hilbert spaces. Additionally, we show that $\mathcal{D}_{(x\otimes y)^n} = \mathcal{D}_{x\otimes y}$ for every $n\in \mathbb{N}$.

\begin{thm}\label{t2.1}
 Let $x,y\in \mathcal{H}$ and $T\in \mathcal{B}(\mathcal{H})$. Then $T\in \mathcal{D}_{x\otimes y}$ if and only if $x\otimes y$  majorizes $(x\otimes T^*y)$ if and only if there exists $M>0$ such that $|\langle z, T^*y\rangle|\leq M |\langle z, y\rangle|$, for all $z\in \mathcal{H}$. Moreover, $\mathcal{D}_{(x\otimes y)^n}=\mathcal{D}_{x\otimes y}$, for every $n\in \mathbb{N}$.
\end{thm}
\begin{proof}
 If $x, y \in \mathcal{H}$ such that $\langle x, y\rangle\neq 0$, then $(x\otimes y)^n=(\langle x,y\rangle)^{n-1} (x\otimes y)$. And so, if $\langle x, y\rangle=0$, then $(x\otimes y)^2=0$ and so it is nilpotent. Hence for the case $\langle x, y\rangle=0$, we have
$$\mathcal{D}_{x\otimes y}=\{T\in \mathcal{B}(\mathcal{H}):\exists M>0, \|(x\otimes y)Tz\|\leq M \|(x\otimes y)z\| \}.$$
And for the case $\langle x, y\rangle\neq 0$,
\begin{align*}
\mathcal{D}_{x\otimes y}&=\{T\in \mathcal{B}(\mathcal{H}):\exists M>0,  \|(x\otimes y)^nTz\|\leq M \|(x\otimes y)^nz\|, \ \ \forall  n\in \mathbb{N}\}\\
&=\{T\in \mathcal{B}(\mathcal{H}):\exists M>0, |\langle x, y\rangle |^{n-1} \|(x\otimes y)Tz\|\leq M |\langle x, y\rangle |^{n-1} \|(x\otimes y)z\|, \ \ \forall  n\in \mathbb{N}\}\\
&=\{T\in \mathcal{B}(\mathcal{H}):\exists M>0, \|(x\otimes y)Tz\|\leq M \|(x\otimes y)z\| \}\\
&=\{T\in \mathcal{B}(\mathcal{H}):\exists M>0,  |\langle Tz, y\rangle|\|x\|\leq M |\langle z, y\rangle|\|x\| \}\\
&=\{T\in \mathcal{B}(\mathcal{H}):\exists M>0,  |\langle z, T^*y\rangle|\leq M |\langle z, y\rangle|\}.\\
\end{align*}

So $T\in \mathcal{D}_{x\otimes y}$ if and only if $(x\otimes y)$ majorizes $(x\otimes T^*y)$ if and only if there exists $M>0$ such that $|\langle z, T^*y\rangle|\leq M |\langle z, y\rangle|$, for all $z\in \mathcal{H}$. By these observations and the fact that $(x\otimes y)^n=(\langle x,y\rangle)^{n-1} (x\otimes y)$, we get that $\mathcal{D}_{(x\otimes y)^n}=\mathcal{D}_{x\otimes y}$, for every $n\in \mathbb{N}$.
\end{proof}

By these observations we get that the Deddens algebra of $x\otimes y$, $\mathcal{D}_{x\otimes y}$, is independent of $x$. This implies that for all $x,y \in \mathcal{H}$ and all $S\in \mathcal{B}(\mathcal{H})$ with $\langle Sx, y\rangle \neq 0$, we have $\mathcal{D}_{x\otimes y}=\mathcal{D}_{Sx\otimes y}$. More generally, for all $x\in \mathcal {H}$ such that $x, z\notin \{y\}^{\perp}$, we have $\mathcal{D}_{x\otimes y}=\mathcal{D}_{z\otimes y}$.\\
From Theorem 2.8 of \cite{lp}, for unit vectors $x,y\in \mathcal{H}$,
\begin{equation}\label{e1}
\mathcal{B}_{x\otimes y}=\{T\in \mathcal{B}(\mathcal{H}): y \ \text{is an eigenvector for} \ T^*\}.
\end{equation}
In the following proposition we aim to characterize elements of Deddens and spectral radius algebras of operators that are similar to rank one operators.
\begin{prop}\label{p2.2}
Let $T, A\in\mathcal{B}(\mathcal{H})$, $x,y\in \mathcal{H}$ and $A$ be invertible such that $T$ is similar to $x\otimes y$ by $A$. Then $S\in \mathcal{D}_{T}$ if and only if $(A^{-1}x\otimes A^*y)$ majorizes $(A^{-1}x\otimes S^*A^*y)$ if and only if there exists $M>0$ such that $|\langle z, S^*A^*y\rangle|\leq M |\langle z, A^*y\rangle|$, for all $z\in \mathcal{H}$. Moreover,
$$\mathcal{B}_{T}=\{S\in \mathcal{B}(\mathcal{H}): A^*y \ \text{is an eigenvector for} \ S^*\}.$$
\end{prop}
\begin{proof}
Since $T$ is similar to $x\otimes y$ by $A$, then $AT=(x\otimes y)A$ and so
$$T=A^{-1}(x\otimes y)A=(A^{-1}x\otimes A^*y).$$
 Hence by Theorem 2.8 of \cite{lp} and Theorem \ref{t2.1} we get the proof.

\end{proof}
In the next lemma we get that if a bounded linear operator is majorized by a rank one operator is rank one.
\begin{lem}\label{p2.3}
Let $x,y\in \mathcal{H}$ and $T\in \mathcal{B}(\mathcal{H})$. If $x\otimes y$ majorizes $T$, then $T$ is a rank one operator and therefore there exists $h\in \mathcal{H}$ such that $T=h\otimes y$.
\end{lem}
\begin{proof}
If $x\otimes y$ majorizes $T$, then by Theorem \ref{t0}, we have
 $$\mathcal{R}(T^*)\subseteq \mathcal{R}(y\otimes x)=\{\alpha y: \alpha\in \mathbb{C}\}.$$

 Hence $T^*$ is a rank one operator and so there exists $h\in \mathcal{H}$ such that $T^*=y\otimes h$. Consequently $T=h\otimes y$. This completes the proof.
\end{proof}
In the next Theorem we characterize Deddense and spectral radius algebras of operators majorized by rank one operators.
\begin{thm}\label{t2.2}
Let $x,y\in \mathcal{H}$ and $T\in \mathcal{B}(\mathcal{H})$. Then if $x\otimes y$ majorizes $T$, then  $S\in \mathcal{D}_{T}$  if and only if $h\otimes y$  majorizes $(h\otimes S^*y)$, for some $h\in \mathcal{H}$ if and only if there exists $M>0$ such that $|\langle z, S^*y\rangle|\leq M |\langle z, y\rangle|$, for all $z\in \mathcal{H}$. Also,
 $$\mathcal{B}_{T}=\{S\in \mathcal{B}(\mathcal{H}): y \ \text{is an eigenvector for} \ S^*\}=\mathcal{B}_{x\otimes y}.$$
\end{thm}
\begin{proof}
Since $x\otimes y$ majorizes $T$, then by the Lemma \ref{p2.3}, there exists $h\in \mathcal{H}$ such that $T=h\otimes y$. Therefore by Theorem \ref{t2.1} we get that
$S\in \mathcal{D}_{T}$  if and only if $h\otimes y$  majorizes $(h\otimes S^*y)$, for some $h\in \mathcal{H}$ if and only if there exists $M>0$ such that $|\langle z, S^*y\rangle|\leq M |\langle z, y\rangle|$, for all $z\in \mathcal{H}$. Also, by \ref{e1} we have
 $$\mathcal{B}_{T}=\{S\in \mathcal{B}(\mathcal{H}): y \ \text{is an eigenvector for} \ S^*\}=\mathcal{B}_{x\otimes y}.$$
\end{proof}
Now we consider the quasi isometry operators on the Hilbert space $\mathcal{H}$ and characterize Deddens and spectral radius algebras of them. The operator $T\in \mathcal{B}(\mathcal{H})$ is called quasi-isometry if $T^*(T^*T)T=T^*T$.
\begin{thm}\label{t2.6}
Let $S,T\in \mathcal{B}(\mathcal{H})$. If $T$ is quasi-isometry, then $S\in \mathcal{D}_T$ if and only if $T$ majorizes $TS$. Moreover, if $r(T)<1$, then $\mathcal{B}_T=\mathcal{B}(\mathcal{H})$. Also, for the case $r(T)\geq 1$, $S\in \mathcal{B}_T$ if and only if there exists $M>0$ such that
$$\|Sx\|+\|TSx\|\alpha_m\leq M(\|x\|+\|Tx\|\alpha_m), \forall m\in \mathbb{N}, \ \ \forall x\in \mathcal{H},$$
where $\alpha_m=\sum^{\infty}_{n=1}d^{2n}_m$.
\end{thm}
\begin{proof}
If $T$ is quasi-isometry, $T^*(T^*T)T=T^*T$, then for every $n\in \mathbb{N}$, $T^{*^n}T^n=T^*T$. This implies that for every $n\in \mathbb{N}$ and $x\in \mathcal{H}$,
$$\|T^nx\|^2=\langle T^nx, T^nx\rangle=\langle T^{*^n}T^nx,x\rangle=\langle T^*Tx,x\rangle=\langle Tx,Tx\rangle=\|Tx\|^2.$$
Let $S\in \mathcal{B}(\mathcal{H})$. Then $\|T^nSx\|=\|TSx\|$, $\|T^nx\|=\|Tx\|$, for each $x\in \mathcal{H}$, $n\in \mathbb{N}$, and so $S\in \mathcal{D}_T$ if and only if there exists $M>0$ such that
$$\|T^nSx\|\leq M\|T^nx\|, \ \ \forall n\in \mathbb{N}, \ \ x\in \mathcal{H}$$
if and only if there exists $M>0$ such that
$$\|TSx\|\leq M\|Tx\|, \ \ \ \ \ \text{for all} \ x\in X.$$
This implies that $S\in \mathcal{D}_T$ if and only if $T$ majorizes $TS$.\\
By our assumptions we have $T^{*^n}T^n=T^*T$, for every $n\in \mathbb{N}$. So $S\in \mathcal{B}_T$ if and only if there exists $M>0$ such that

$$\sum^{\infty}_{n=0}d^{2n}_m\|T^nSx\|\leq M\sum^{\infty}_{n=0}d^{2n}_m\|T^nx\|, \forall m\in \mathbb{N}, \ \ \forall x\in \mathcal{H}$$
if and only if
$$\sum^{\infty}_{n=0}d^{2n}_m\|TSx\|\leq M\sum^{\infty}_{n=0}d^{2n}_m\|Tx\|, \forall m\in \mathbb{N}, \ \ \forall x\in \mathcal{H}$$
if and only if
$$\|Sx\|+\|TSx\|\sum^{\infty}_{n=1}d^{2n}_m\leq M(\|x\|+\|Tx\|\sum^{\infty}_{n=1}d^{2n}_m), \forall m\in \mathbb{N}, \ \ \forall x\in \mathcal{H}.$$
By definition, $\{d_m\}$ is an increasing sequence convergent to $\frac{1}{r(T)}$, indeed
 $$\sup_{m\in \mathbb{N}}d_m=\frac{1}{r(T)}.$$
  Since $T$ is a bounded operator, then $r(T)<\infty$. Hence the series $\alpha_m=\sum^{\infty}_{n=1}d^{2n}_m$ is convergent for all $m\in \mathbb{N}$ if and only if $r(T)\geq1$, otherwise it is divergent. Hence for the case $r(T)\leq1$,
$S\in \mathcal{B}_T$ if and only if there exists $M>0$ such that
$$\|Sx\|+\|TSx\|\alpha_m\leq M(\|x\|+\|Tx\|\alpha_m), \forall m\in \mathbb{N}, \ \ \forall x\in \mathcal{H}.$$
Also, for the case $r(T)<1$, $\sum^{\infty}_{n=1}d^{2n}_m=\infty$ and consequently $\mathcal{B}_T=\mathcal{B}(\mathcal{H})$.
\end{proof}

\section{ \sc\bf Applications to WCT operators}
Let $(X, \mathcal{F}, \mu)$ be a complete $\sigma$-finite measure space. All statements regarding sets and functions are to be interpreted as holding true except for sets of measure zero.

For a $\sigma$-subalgebra $\mathcal{A}$ of $\mathcal{F}$, the conditional expectation operator associated with $\mathcal{A}$ is the mapping $f \rightarrow E^{\mathcal{A}}f$, defined for all non-negative $f$ as well as for all $f \in L^2(\mathcal{F}) = L^2(X, \mathcal{F}, \mu)$. Here, $E^{\mathcal{A}}f$ is the unique $\mathcal{A}$-measurable function that satisfies the equation:

$$\int_{A}(E^{\mathcal{A}}f)d\mu = \int_{A}fd\mu \ \ \ \ \ \ \ \forall A\in \mathcal{A}.$$

We will often use the notation $E$ instead of $E^{\mathcal{A}}$. This operator will play a significant role in our
 work, and we list some of its useful properties here:
%

\vspace*{0.2cm} \noindent $\bullet$ \  If $g$ is
$\mathcal{A}$-measurable, then $E(fg)=E(f)g$.

\noindent $\bullet$ \ If $f\geq 0$, then $E(f)\geq 0$; if $E(|f|)=0$,
then $f=0$.

\noindent $\bullet$ \ $|E(fg)|\leq
(E(|f|^2))^{\frac{1}{2}}(E(|g|^{2}))^{\frac{1}{2}}$.

\noindent $\bullet$ \ For each $f\geq 0$, $z(E(f))$ is the smallest $\mathcal{A}$-set containing $z(f)$, where $z(f)=\{x\in X: f(x)\neq 0\}$.

\vspace*{0.2cm}\noindent A detailed discussion and verification of
most of these properties may be found in \cite{rao}.

\begin{defn}
Let $(X,\mathcal{F},\mu)$ be a $\sigma$-finite measure space and $\mathcal{A}$ be a
$\sigma$-sub-algebra of $\mathcal{F}$ such that $(X,\mathcal{A},\mu_{\mathcal{A}})$ is also $\sigma$-finite. Let $E$ be the conditional
expectation operator relative to $\mathcal{A}$. If $u,w:X\rightarrow \mathbb{C}$ are $\mathcal{F}$-measurable functions such that $uf$ is conditionable (i.e., $E(uf)$ exists) and $wE(uf)\in L^{2}(\mathcal{F})$ for all $f\in L^{2}(\mathcal{F})$, then the corresponding weighted conditional type (or briefly WCT) operator is the linear transformation $M_wEM_u:L^2(\mathcal{F})\rightarrow L^{2}(\mathcal{F})$ defined by $f\rightarrow wE(uf)$.
\end{defn}
As it was proved in \cite{ej}, the WCT operator $M_wEM_u$ on $L^2(\mathcal{F})$ is bounded if and only if $(E(|u|^{2}))^{\frac{1}{2}}(E(|w|^2))^{\frac{1}{2}}\in
L^{\infty}(\mathcal{A})$.

Now un the next theorem we characterize Deddens algebra of WCT operator $T=M_{\bar{u}}EM_u$.

\begin{thm}\label{t3.22}
Let $T=M_{\bar{u}}EM_u$ and $S\in \mathcal{B}(L^2(\mathcal{F})$. Then $S\in \mathcal{D}_T$ if and only if $PSP=PS$ and $XP=PSP\in \mathcal{D}_{M_{E(|u|^2)}}$, in which $P=P_{\mathcal{N}(EM_u)^{\perp}}$.
\end{thm}
\begin{proof}
We consider the Hilbert space $L^2(\mathcal{F})$ as a direct sum $\mathcal{H}_1\oplus \mathcal{H}_2$, in which
 $$\mathcal{H}_2=\mathcal{N}(EM_u)=\{f\in L^2(\mathcal{F}): E(uf)=0\}$$
and
 $$\mathcal{H}_1=\mathcal{H}^{\perp}_2=\overline{\bar{u}L^2(\mathcal{A})}.$$

Easily we get that $\mathcal{N}(EM_u)=\mathcal{N}(M_{\bar{u}}EM_u)$, because
$$\langle M_{\bar{u}}EM_uf,f\rangle=\|EM_uf\|^2.$$
Since $T=M_{\bar{u}}EM_u$ is bounded, then $E(|u|^2)\in L^{\infty}(\mathcal{A})$ and so $E(uf)\in L^2(\mathcal{A})$, for all $f\in L^2(\mathcal{A})$. This implies that $TP=M_{E(|u|^2)}P$ and consequently $T^nP=M_{(E(|u|^2))^n}P$, for every $n\in \mathbb{N}$. Thus the corresponding block matrix of $T=M_{\bar{u}}EM_u$ is
\begin{center}
$ T^n=\left(
 \begin{array}{cc}
     M_{(E(|u|^2))^n} & 0 \\
     0 & 0 \\
   \end{array}
 \right)
$
and also for any $S\in \mathcal{B}(\mathcal{H})$,
$ S=\left(
   \begin{array}{cc}
     X & Y \\
     Z & W \\
   \end{array}
 \right)
$.
\end{center}
Hence for every $f\in L^2(\mathcal{F})$, we have $f=Pf+(P^{\perp}f)=\left(
   \begin{array}{cc}
     Pf \\
     P^{\perp}f \\
   \end{array}
 \right)$, in which $P=P_{\mathcal{H}_1}$ and $P^{\perp}=I-P$. So we have
\begin{center}
$ T^nSf=
 \left(
   \begin{array}{cc}
      M_{(E(|u|^2))^n}XPf+M_{(E(|u|^2))^n}YP^{\perp}f \\
     0 \\
   \end{array}
 \right)
$.
\end{center}
Therefore $S\in \mathcal{D}_T$ if and only if there exists $M>0$ such that

$$\|M_{(E(|u|^2))^n}XPf+M_{(E(|u|^2))^n}YP^{\perp}f\|\leq M\|M_{(E(|u|^2))^n}f\|, \ \ \ \forall n\in \mathbb{N}, \ \ f\in L^2(\mathcal{F}.$$

By Theorem 2.4 of \cite{ej1} we have $S\in \mathcal{B}_T$ if and only if $\mathcal{N}(EM_u)$ is invariant under $S$, and so $S\in \mathcal{B}_T$ if and only if $P^{\perp}SP^{\perp}=SP^{\perp}$ if and only if $PSP=PS$. Moreover, we know that $\mathcal{D}_T\subseteq \mathcal{B}_T$. This means that if $S\in \mathcal{D}_T$, then it has to have the property $PSP=PS$. So we have $Y=0$. Also,
$$XP=PSP^2=PSP=P^2SP=PX=PXP.$$
 Hence $S\in \mathcal{D}_T$ if and only if there exists $M>0$ such that

$$\|M_{(E(|u|^2))^n}PSPf\|\leq M\|M_{(E(|u|^2))^n}f\|, \ \ \ \forall n\in \mathbb{N}, \ \ f\in L^2(\mathcal{F}.$$

 By all these observations we get that $S\in \mathcal{D}_T$ if and only if $PSP=PS$ and $XP=PSP\in \mathcal{D}_{M_{E(|u|^2)}}$.
\end{proof}
Here we have the next corollary.
\begin{cor}\label{cor3.1}
Let $T=M_{a\bar{u}}EM_u\in \mathcal{B}(L^2(\mathcal{F})$, $a$ be an $\mathcal{A}$-measurable function and $S\in \mathcal{B}(L^2(\mathcal{F})$. Then $S\in \mathcal{D}_T$ if and only if $PSP=PS$ and $XP=PSP\in \mathcal{D}_{M_{aE(|u|^2)}}$, in which $P=P_{\mathcal{N}(EM_u)^{\perp}}$.
\end{cor}

Let $S=S(E(|u|^2))$, $S_0=S(E(u))$, $G=S(E(|w|^2))$, $G_0=S(w)$, $F=S(E(uw))$. By the conditional type H$\ddot{o}$lder inequality we get that
$F\subseteq S\cap G$, $S(wE(uf))\subseteq S\cap G$, for all $f\in L^2(\mu)$, and also by the elementary properties of the conditional expectation $E$ we have $S_0\subseteq S$ and $G_0\subseteq G$.\\

For WCT operator $T=M_wEM_u$, as a bounded linear operator on the Hilbert space $L^2(\mu)$, we have $T^*=M_{\bar{u}}EM_{\bar{w}}$ and the following properties hold \cite{ej},

$$T^*T=M_{\bar{u}E(|w|^2)}EM_u, \ \ \ \ \ \ \ TT^*=M_{wE(|u|^2)}EM_{\bar{w}},$$
 $$TT^*T=M_{E(|u|^2)E(|w|^2)}M_{w}EM_{u}=M_{E(|u|^2)E(|w|^2)}T, \ \ \ \ \ \ \ T^*TT=M_{E(wu)E(|w|^2)}M_{\bar{u}}EM_u.$$

\begin{prop}\label{p3.1}
The WCT operator $T=M_wEM_u:L^2(\mu)\rightarrow L^2(\mu)$ is quasinormal if and only if $T(f)=M_{v\bar{u}}EM_u$,
 in which $v=\frac{E(uw)}{E(|u|^2)}\chi_{S\cap G}$.
 \end{prop}
 \begin{proof}
The WCT operator $T=M_wEM_u$ is quasinormal if and only if $TT^*T=T^*TT$ if and only if
$$E(|u|^2)E(|w|^2)wE(uf)=E(wu)E(|w|^2)\bar{u}E(uf), \ \ \ \ \text{for all} \ \ \ f\in L^2(\mu).$$
Since $T$ is bounded, then $E(|u|^2)E(|w|^2)\in L^{\infty}(\mathcal{A})$ and $\|T\|=\|(E(|u|^2))^{\frac{1}{2}}(E(|w|^2)^{\frac{1}{2}}\|_{\infty}$, \cite{ej}.
 By the fact that $(X,\mathcal{A}, \mu_{\mathcal{A}})$ is a $\sigma$-finite measure space, we have an increasing sequence  $\{A_n\}_{n\in \mathbb{N}}\subseteq \mathcal{A}$, with $0<\mu(A_n)<\infty$ and $X=\cup_{n\in \mathbb{N}}A_n$. Now we set $f_n=\bar{u}\sqrt{E(|w|^2)}\chi_{A_n}$, for every $n\in \mathbb{N}$. So
 \begin{align*}
 \|f_n\|^2_{2}&=\int_X|u|^2E(|w|^2)\chi_{A_n}d\mu\\
 &=\int_XE(|u|^2)E(|w|^2)\chi_{A_n}d\mu\\
 &\leq \|E(|u|^2)E(|w|^2)\|^2_{\infty}\mu(A_n)\\
 &<\infty,
\end{align*}
and hence $f_n\in L^2(\mu)$, for all $n\in \mathbb{N}$.\\

Suppose that $T=M_wEM_u$ is quasinormal, then by the above observations we have
\begin{align*}
E(|u|^2)E(|w|^2)wE(|u|^2)\sqrt{E(|w|^2)}\chi_{A_n}&=E(|u|^2)E(|w|^2)wE(uf_n)\\
&=E(wu)E(|w|^2)\bar{u}E(uf_n)\\
&=E(wu)E(|w|^2)\bar{u}E(|u|^2)\sqrt{E(|w|^2)}\chi_{A_n}
\end{align*}
Moreover, by taking limit we get
\begin{align*}
E(|u|^2)E(|w|^2)wE(|u|^2)\sqrt{E(|w|^2)}&=\lim _{n\rightarrow \infty}E(|u|^2)E(|w|^2)wE(|u|^2)\sqrt{E(|w|^2)}\chi_{A_n}\\
&=\lim _{n\rightarrow \infty}E(wu)E(|w|^2)\bar{u}E(|u|^2)\sqrt{E(|w|^2)}\chi_{A_n}\\
&=E(wu)E(|w|^2)\bar{u}E(|u|^2)\sqrt{E(|w|^2)}.
\end{align*}
So we have
$$E(|u|^2)E(|w|^2)wE(|u|^2)\sqrt{E(|w|^2)}=E(wu)E(|w|^2)\bar{u}E(|u|^2)\sqrt{E(|w|^2)}.$$
Therefore
$$wE(|u|^2)\chi_{S\cap G}=E(uw)\bar{u}\chi_{S\cap G},$$
and hence $w=\frac{E(uw)}{E(|u|^2)}\bar{u}\chi_{S\cap G}$.
 Consequently we have
 $$T(f)=wE(uf)=w\chi_{S\cap G} E(uf)=\frac{E(uw)}{E(|u|^2)}\chi_{S\cap G}\bar{u}E(uf)=M_{v\bar{u}}EM_u,$$
 in which $v=\frac{E(uw)}{E(|u|^2)}\chi_{S\cap G}$.
 \end{proof}
 In the next theorem we characterize Deddens algebra of quasinormal WCT operators.
 \begin{thm}
 Let WCT operator $T=M_wEM_u:L^2(\mu)\rightarrow L^2(\mu)$ be quasinormal and $S\in \mathcal{B}(L^2(\mathcal{F})$. Then $S\in \mathcal{D}_T$ if and only if $PSP=PS$ and $XP=PSP\in \mathcal{D}_{M_{vE(|u|^2)}}$, in which $P=P_{\mathcal{N}(EM_u)^{\perp}}$ and $v=\frac{E(uw)}{E(|u|^2)}\chi_{S\cap G}$..
 \end{thm}
 \begin{proof}
 It is a direct consequence of Corollary \ref{cor3.1} and Proposition \ref{p3.1}.
 \end{proof}
 Here we provide two technical lemmas for later use.
\begin{lem}\label{l1}
Let $g\in L^{\infty}( \mathcal{A})$
and let $T:L^{2}(\Sigma)\rightarrow L^{2}(\Sigma)$ be the WCT operator $T=M_wEM_u$. Then $M_gT=0$ if and only if $g=0$ on
$S(E(|w|^{2})E(|u|^{2}))=S\cap G$.\\
 \end{lem}
\begin{proof}
 By Theorem 2.1 of \cite{ej} we have
$$\|M_gT\|^{2}=\||g|^{2}E(|w|^{2})E(|u|^{2})\|_{\infty},$$.

Hence $M_gT=0$ if and only if
 $$\|M_gT\|^{2}=|g|^{2}E(|w|^{2})E(|u|^{2})=0$$
if and only if $g=0$ on
$S(E(|w|^{2})E(|u|^{2}))$.
\end{proof}
 As we now from \cite{es} the Moore-Penrose inverse of WCT operator $T=M_wEM_u$ is
  $$T^{\dagger}=M_{\frac{\chi_{S\cap G}}{E(|u|^2)E(|w|^2)}}M_{\bar{u}}EM_{\bar{w}}=M_{\frac{\chi_{S\cap G}}{E(|u|^2)E(|w|^2)}}T^*.$$

  \begin{lem}\label{l2}
   Let $T=M_wEM_u$. Then $T^{\dagger}=T^*$ if and only if  $E(|u|^2)E(|w|^2)=\chi_{S\cap G}$.
  \end{lem}
  \begin{proof}
  It is obvious that
  $$T^{\dagger}=M_{\frac{\chi_{S\cap G}}{E(|u|^2)E(|w|^2)}}T^*$$
  and so $T^{\dagger}=T^*$ if and only if $(1-M_{\frac{\chi_{S\cap G}}{E(|u|^2)E(|w|^2)}}))T^*=0$.
  Therefore by the Lemma \ref{l1} we get that $T^{\dagger}=T^*$ if and only if $E(|u|^2)E(|w|^2)=1$, $\mu$, a.e., on $S\cap G$ if and only if  $E(|u|^2)E(|w|^2)=\chi_{S\cap G}$, $\mu$, a.e., on $S\cap G$.
  \end{proof}
  We recall that $T\in \mathcal{B}(\mathcal{H})$ is partial isometry if and only if $TT^*T=T$. From Theorem 3.2 of \cite{ej} we have $T=M_wEM_u$ is partial isometry if and only if $E(|u|^2)E(|w|^2)=\chi_{S\cap G}$, $\mu$, a.e., on $S\cap G$. Now we have the following corollary.
\begin{cor}
Let $T=M_wEM_u\in \mathcal{B}(\mathcal{H})$. Then $T$ is partial isometry if and only if  $T^{\dagger}=T^*$.
\end{cor}

%
%

 Now in the next Theorem we characterize quasi-isometry WCT operators.
\begin{thm}\label{t3.5}
Let $T=M_uEM_w$. Then For each $n\in \mathbb{N}$, $T$ is $n$-quasi-isometry if and only if $T$ is 1-quasi-isometry if and only if $|E(uw)|= 1$, $\mu$, a.e.
\end{thm}
\begin{proof}
Let $T=M_wEM_u$. Then for each $n\in \mathbb{N}$,

$$T^{*^n}T^n=T^{*^{n+1}}T^{n+1}$$
if and only if
$$M_{E(|w|^2)|E(uw)|^{2(n-1)}}M_{\bar{u}}EM_u=M_{E(|w|^2)|E(uw)|^{2(n)}}M_{\bar{u}}EM_u$$
if and only if
$$M_{E(|w|^2)|E(uw)|^{2(n-1)}(1-|E(uw)|^2)}M_{\bar{u}}EM_u=0.$$
Since $E(|w|^2)|E(uw)|^{2(n-1)}(1-|E(uw)|^2)$ is an $\mathcal{A}$-measurable function, then by the Lemma \ref{l1} we get that
$$M_{E(|w|^2)|E(uw)|^{2(n-1)}(1-|E(uw)|^2)}M_{\bar{u}}EM_u=0$$ if and only if $E(|w|^2)|E(uw)|^{2(n-1)}(1-|E(uw)|^2)=0$, $\mu$, a.e., on $S$ if and only if $|E(uw)|=1$, $\mu$, a.e., on $F=S(E(uw))$ if and only if $|E(uw)|=1$, $\mu$, a.e.\\
Moreover, for $n=1$, $T^{*}T=T^{*^{2}}T^{2}$ if and only if $E(|w|^2)(1-|E(uw)|^2)=0$, $\mu$, a.e., on $S$, if and only if $1-|E(uw)|^2=0$, $\mu$, a.e., on $S\cap G$. Since $F\subseteq S\cap G$, then $1-|E(uw)|^2=0$, $\mu$, a.e., on $S\cap G$ if and only if $1-|E(uw)|^2=0$, $\mu$, a.e., on $F$ if and only if $|E(uw)|=1$, $\mu$, a.e.
\end{proof}


\begin{thm}\label{t3.6}
Let $T=M_wEM_u$ be WCT operators on $L^2(\mu)$,  $|E(uw)|=1$, a.e., on $F=S(E(uw)$ and $S\in \mathcal{B}(L^2(\mu))$. Then $S\in \mathcal{D}_{T}$ if and only if $T$ majorizes $TS$. Also, $S\in \mathcal{B}_T$ if and only if there exists $M>0$ such that
$$\|Sf\|+\|TSf\|\alpha_m\leq M(\|f\|+\|Tf\|\alpha_m), \forall m\in \mathbb{N}, \ \ \forall f\in L^2(\mu),$$
where $\alpha_m=\sum^{\infty}_{n=1}d^{2n}_m$.
\end{thm}
\begin{proof}
We know that $T^nf=E(uw)^{n-1}Tf$, for all $n\in \mathbb{N}$ and $f\in L^2(\mu)$. If $|E(uw)|=1$, a.e., on $F=S(E(uw)$, then $\|T^nf\|=\|Tf\|$, for all $n\in \mathbb{N}$ and $f\in L^2(\mu)$. Hence for $S\in \mathcal{B}(L^2(\mu))$, we have $S\in \mathcal{D}_{T}$ if and only if $T$ majorizes $TS$. Similarly, we get that $r(T)=1$ and so by Theorem \ref{t2.6}, $S\in \mathcal{B}_T$ if and only if there exists $M>0$ such that
$$\|Sf\|+\|TSf\|\alpha_m\leq M(\|f\|+\|Tf\|\alpha_m), \forall m\in \mathbb{N}, \ \ \forall f\in L^2(\mu).$$
\end{proof}
Now by Theorems \ref{t3.5} and \ref{t3.6} we have the following corollary.
\begin{cor}
Let $T=M_wEM_u$ be WCT operators on $L^2(\mu)$ and $S\in \mathcal{B}(L^2(\mu))$. If $T$ is $1$-quasi-isometry or $n$-quasi-isometry, for every $n\in \mathbb{N}$ or there exists $n\in \mathbb{N}$ such that $T$ is $n$-quasi-isometry, then $S\in \mathcal{D}_{T}$ if and only if $T$ majorizes $TS$. And, $S\in \mathcal{B}_T$ if and only if there exists $M>0$ such that
$$\|Sx\|+\|TSx\|\alpha_m\leq M(\|x\|+\|Tx\|\alpha_m), \forall m\in \mathbb{N}, \ \ \forall x\in \mathcal{H},$$
where $\alpha_m=\sum^{\infty}_{n=1}d^{2n}_m$.
\end{cor}
\begin{thm}\label{t3.9}
Let $f_1,f_2\in L^2(\mu)$ and $T=M_wEM_u$. Then if $f_1\otimes f_2$ majorizes $T$, then  $S\in \mathcal{D}_{T}$  if and only if $g\otimes f_2$  majorizes $(g\otimes S^*f_2)$, for some $g\in L^2(\mu)$ if and only if there exists $M>0$ such that $\int_X h\overline{S^*(f_2)}d\mu\leq M \int_X h\overline{f_2}d\mu$, for all $h\in L^2(\mu)$. Also,
 $$\mathcal{B}_{T}=\{S\in \mathcal{B}(L^2(\mu)): h \ \text{is an eigenvector for} \ S^*\}=\mathcal{B}_{f\otimes g}.$$
\end{thm}
\begin{proof}
As is known the inner product of the Hilbert space $L^2(\mu)$ is as:
$$\langle f, g\rangle=\int_Xf\bar{g}d\mu, \ \ \ \text{for all}, \ \ f,g \in L^2(\mu).$$
 Hence by Theorem \ref{t2.2} we get the proof.
\end{proof}

\section{ \sc\bf Declarations }
\textbf{Ethical Approval}
Not applicable.

\textbf{Competing interests}
The authors declare that there is no conflict of interest.


\textbf{Availability of data and materials}
Our manuscript has no associate data.

\end{document}